%% file: Rel-gr-width-arXiv-3.tex
\documentclass[11pt]{amsart}

\usepackage{amsthm, amsfonts, graphicx, pinlabel}
\usepackage[all]{xy}
\usepackage{hyperref}
\usepackage[usenames,dvipsnames]{xcolor}

\include{header}  

\newcommand{\dfn}[1]{{\textbf {#1}}}
\newcommand{\leg}{\ensuremath{\Lambda}}

\newcommand{\lag}[2]{\ensuremath{{L}_{#1}^{#2}}}
\newcommand{\alg}{\ensuremath{\mathcal{A}}}

\newcommand{\aug}{\ensuremath{\varepsilon}}

\DeclareMathOperator{\im}{Im}
\DeclareMathOperator{\Aug}{Aug}
\newcommand{\cont}[1]{\ensuremath{\mathcal{C}(#1)}}
\def\mc{\underline{c}}
\def\Mc{\overline{c}}
\def\PR{\mathcal C(P)}
\def\Rho{P}
\def\tL{\widetilde L}
\def\tleg{\widetilde \leg}
\def\taug{\widetilde \aug}
\def\tlambda{\overline \lambda}

\def\trho{\widetilde \rho}
\def\tb{\widetilde{\mathbf{b}}}
\def\tC{\widetilde{C}}
\def\tP{\widetilde{P}}

\def\cobord{fundamental cobordism}
\newcommand{\mixed}[1]{\ensuremath{\overset{{}_\leftarrow}{#1}}}

\begin{document}

\title[Relative Gromov Width of Lagrangian Cobordisms]
{The Relative Gromov Width of Lagrangian Cobordisms {between Legendrians}} 

\author[J. Sabloff]{Joshua M. Sabloff} \address{Haverford College,
Haverford, PA 19041} \email{jsabloff@haverford.edu} \thanks{JS is
partially supported by NSF grant DMS-1406093, and the Schmidt  Fellowship at IAS.  Both JS and LT 
received support at IAS from The Fund for Mathematics.}

\author[L. Traynor]{Lisa Traynor} \address{Bryn Mawr College, Bryn
Mawr, PA 19010} \email{ltraynor@brynmawr.edu} 

\begin{abstract}
We obtain upper and lower bounds for the relative Gromov width 
of 
  Lagrangian cobordisms between Legendrian submanifolds.
  Upper bounds arise from the existence of $J$-holomorphic disks   with boundary on the Lagrangian cobordism that
  pass through the center of a given symplectically embedded ball.  The areas of these disks ---
	  and hence the sizes of these balls --- are
	controlled by a real-valued fundamental capacity, a quantity derived from the algebraic structure of filtered  linearized Legendrian Contact Homology of the
	 Legendrian at the top of the cobordism.  Lower bounds come from explicit constructions that use neighborhoods of 
	Reeb chords in the Legendrian ends.  We also study relationships between the relative Gromov width and another quantitative measurement,
	the  length of a cobordism between two Legendrian submanifolds.
\end{abstract}

\date{\today}

\maketitle

\section{Introduction}
\label{sec:intro}

\subsection{Quantitative Study of Lagrangian Cobordisms}
\label{ssec:context}

    In \cite{josh-lisa:cob-length}, the authors inaugurated the study of quantitative questions about Lagrangian cobordisms between Legendrian submanifolds by defining and investigating the \emph{lengths} of such cobordisms.  We continue the investigation of quantitative features of Lagrangian cobordisms through the study of relative Gromov \emph{widths}, a relative version of a classical symplectic measurement.  The key results in this paper encompass upper bounds
    on relative Gromov widths
     derived from Floer-theoretic {techniques},
    lower bounds  arising from constructions,   
 and relationships between the quantitative measures of length and width.

The background for these questions and results begins with Gromov's seminal non-squeezing theorem.  Gromov's proof in \cite{gromov:hol} relies on the non-triviality of the \dfn{Gromov 
width} of a symplectic manifold $(X,\omega)$.  The width of a symplectic manifold $(X, \omega)$
 is the supremum of the quantities $\pi r^2$ taken over all symplectic embeddings of closed balls of radius $r$ into $(X, \omega)$.
In the presence of a Lagrangian submanifold $L \subset X$, Barraud and Cornea \cite{bc:lagr-intersection} defined the \dfn{relative Gromov width} to be the supremum of the quantities $
\pi r^2$ taken over all \dfn{symplectic embeddings (of a ball) relative to $L$}, i.e., symplectic embeddings  $\psi: (B^{2n}(r), \omega_0) \hookrightarrow (X,\omega)$ with the property that $\psi^{-1}(L) = B^{2n}(r) \cap \rr^n$, where $\rr^n$ denotes a 
Lagrangian plane in $\rr^{2n}$. The notation 
  $\psi: B^{2n}(r) \hookrightarrow (X,L)$ will be used to denote a  symplectic embedding relative to $L$.
  
  Though interesting in its own right, the relative Gromov has also played a role in detecting other symplectic phenomena.   In  recent work of Cornea and Shelukhin \cite{cs:cob-metric}, for example, the
relative Gromov width of a pair of Lagrangians in $(M, \omega)$ is used to show that the ``shadow'' measure of Lagrangian cobordisms between
Lagrangians  in $(\cc \times M, \omega_0 \oplus \omega)$ defines a metric and pseudo-metric on  appropriate spaces of Lagrangians in $M$.

Known upper bounds on the relative Gromov width come from finding a $J$-holomorphic curve of bounded area passing through the center of a
given embedded ball.  
For closed Lagrangians, such disks are known to exist when the Lagrangian is monotone \cite{bc:lagr-intersection,bc:uniruling,charette:chv}, is an orientable surface \cite{charette:width}, 
or admits a metric of non-positive sectional curvature \cite{bml}; see also \cite{albers:extrinsic, damian:universal}. Finding appropriate $J$-holomorphic curves is also the key technique in the proof that the shadow of a Lagrangian cobordism between Lagrangians is an upper bound for the relative Gromov width \cite{cs:cob-metric}.  As shown in \cite{rizell:uniruled, murphy:closed-exact}, however, such 
disks do not always exist for closed Lagrangians in symplectizations, and the relative Gromov width may be infinite. Constructions leading to sharp lower bounds on the relative Gromov 
width are more rare: see, for example, \cite{buhovsky:max-pack}.

The goal of this paper is to extend the calculation of the relative Gromov width to \dfn{exact Lagrangian cobordisms between Legendrian submanifolds}:
we always consider \emph{closed}
Legendrians   in the contactization of a Liouville manifold,  $(\cont{P}, \ker \alpha)$, and cobordisms
are always \emph{properly embedded, orientable, Maslov zero,  exact} Lagrangian submanifolds
of the symplectization $\left(\rr \times \cont{P}, d(e^s \alpha)\right)$ that coincide with cylinders over Legendrians $\leg_\pm$ in the complement of $[s_-,s_+] \times \cont{P}$.  Formal definitions
can be found in Section~\ref{sec:background}.  
 We denote by $L_a^b$ the portion of a Lagrangian cobordism $L$ whose symplectization coordinate lies between $a$ and $b$:
$$L_a^b  = L \cap \left((a,b) \times \cont{P} \right).$$
The \dfn{relative (Gromov) width} of $L_a^b$ is then defined in terms of finding relative symplectic embeddings:
\[
w\left(L_a^b\right) = \sup \left\{ \pi r^2 \;\mid\; \exists\, \psi: B^{2n}(r) \hookrightarrow \left((a,b) \times \cont{P}, L_a^b\right) \right\}.
\]
 Since the relative Gromov width of the ``top half'' $L_a^\infty$ of any Lagrangian cobordism $L$ is infinite (see Lemma~\ref{lem:infinite-top}), we focus attention to the relative Gromov width of the ``bottom half'' $L_{-\infty}^0$, where $L$ is cylindrical outside of $[s,0]$, for $s < 0$.

\subsection{Upper Bounds}
\label{ssec:upper-intro}

Our derivation of an upper bound on the relative Gromov width of a Lagrangian cobordism $L$ from $\leg_-$ to $\leg_+$ will follow the now-standard approach of finding a $J$-holomorphic curve of controlled area through the center of a given relative symplectic embedding of a ball. To guarantee the existence of an appropriate $J$-holomorphic curve, we will assume that $\leg_+$ is connected,  $\leg_-$ and $\leg_+$ are horizontally displaceable, and $\leg_{-}$ admits an augmentation $\aug_-$; we term such a cobordism a \dfn{\cobord} and define it officially in Definition~\ref{defn:fun-cobord}. Note that a ``horizontally displaceable'' Legendrian $\leg \subset \cont{P}$ is one whose Lagrangian projection of $\leg$ to $P$ is displaceable by a Hamiltonian isotopy \cite{high-d-duality}; in particular, any $\leg \subset \rr^{2n+1} = J^1\rr^n$ is horizontally displaceable.  Under these assumptions, we may relate the generator of $H_0(L)$ with the fundamental class of $\leg_+$ using the Generalized Duality Long Exact Sequence of \cite[{Theorem 1.2}]{c-dr-g-g-cobordism}; see Theorem~\ref{thm:duality-general}. 
 An examination of this relationship at the chain level leads to the desired $J$-holomorphic curve; the area of the curve will be governed by 
  the  
\dfn{fundamental capacity} $c(\leg_{+}, \aug_{+})$ of    $\leg_{+}$, where the augmentation $\aug_+$ of $\leg_+$ is induced from $\aug_-$ via $L$.

The upper  bound to the width will be given in terms of the minimal and maximal fundamental capacities of the  Legendrian $\leg_+$ with respect
to sets of augmentations.  Let $\Aug(\leg_{+})$ be the set of all augmentations of the 
Legendrian contact homology differential graded algebra of $\leg_{+}$ (see Section~\ref{ssec:lch}), and given a Lagrangian cobordism $L$ from $\leg_-$ to $\leg_{+}$, 
let $\Aug_L(\leg_{+}) \subseteq \Aug(\leg_{+})$ be the set of augmentations of $\leg_{+}$ induced from augmentations of the Legendrian $\leg_-$;
see Remark~\ref{rem:induced-aug}.  We now define the the {\bf miniumum $L$-induced fundamental capacity} and the 
{\bf maximum fundamental capacity}, respectively, as:
\begin{align*}
	\mc_L(\leg_{+}) &= \min\{ c(\leg_{+},\aug_{+})\;:\; \aug_{+} \in \Aug_L(\leg_{+}) \}, \\
	\Mc(\leg_{+}) &= \max\{ c(\leg_{+},\aug_{+})\;:\; \aug_{+} \in \Aug(\leg_{+}) \}.
\end{align*}
Observe that for all $L$, we have
$$\mc_L(\leg_{+}) \leq \Mc(\leg_{+}).$$ 
In particular, if  $U^n(r)$ denotes the standard $n$-dimensional
Legendrian unknot in $\rr^{2n+1}$ with a single Reeb chord of height $r$, then
$\mc_{L}(U^{n}(r)) = \Mc(U^{n}(r)) = r$. To the authors' knowledge, there are no known examples where the fundamental capacity depends on the augmentation, though such examples are theoretically possible.

We are now ready to state our main theorem for an upper bound on the relative Gromov width:
  
\begin{thm} \label{thm:ub}  If $L \subset \rr \times \PR$ is a \cobord,   then 
\begin{equation}
w\left( L_{-\infty}^0\right)  \leq 2\mc_{L}(\leg_{+}) \leq 2\Mc(\leg_+).
\end{equation}
\end{thm}

 \subsection{Lower Bounds}  
\label{ssec:lower-intro} To complement the upper bounds on the relative Gromov width in Theorem~\ref{thm:ub},  we derive  lower bounds  through 
the construction of relative symplectic embeddings.  These embeddings are constructed in  a neighborhood of
a Reeb chord  of a Legendrian at an end of the Lagrangian cobordism.  These Reeb chords need to be sufficiently ``extendable'':  a Reeb chord  $\gamma$ of a Legendrian $\leg$ is \dfn{frontwise doubly extendable} if the front projection of the upward (or downward) extension of the Reeb chord to twice its height only intersects the front projection of the Legendrian at $\partial \gamma$;  see Definition~\ref{defn:extendable}.

\begin{thm} \label{thm:lower-isolated-chord}
If $\leg \subset J^1M$ has a frontwise doubly extendable Reeb chord of height $h$, then
\[2h \leq w\left( (\rr \times \leg )_{-\infty}^0\right).\]
\end{thm}

Additional lower bounds  for the width of Lagrangian cobordisms 
come from comparisons to the
widths of the cylindrical ends.

\begin{thm}  \label{thm:neg-end} Suppose $L \subset \rr \times \PR$ is a Lagrangian cobordism from $\leg_-$ to $\leg_+$ that is
cylindrical outside $[s_-,0]$.  Then
\begin{equation} \label{eqn:neg-end}
e^{s_-} w\left(\left(\rr \times \leg_-\right)_{-\infty}^0\right) \leq w(L_{-\infty}^0).
\end{equation}
If $L$ is cylindrical outside $[s_-, -\epsilon]$, for some $\epsilon > 0$, then 
\begin{equation} \label{eqn:collared}
w\left((\rr \times \leg_+)_{-\infty}^0 \right)\leq w(L_{-\infty}^0).
\end{equation}
\end{thm}

Theorems~\ref{thm:ub}, \ref{thm:lower-isolated-chord}, and \ref{thm:neg-end} combine to give precise calculations of some fundamental cobordisms
that are collared near the top.

\begin{cor} \label{cor:calc}   Suppose $L \subset \rr \times \PR$ is a fundamental cobordism that is cylindrical outside $[s_-, -\epsilon]$, for some $\epsilon > 0$.  If the longest Reeb chord of $\leg_{+}$ has height $r$ and is frontwise doubly extendable, then
$$ w\left(  L_{-\infty}^0 \right)= 2\mc_{L}(\leg_{+}) = 2\Mc(\leg_{+}) = 2r. $$
\end{cor}

\begin{proof}  Combining Theorem~\ref{thm:lower-isolated-chord} and Equation~(\ref{eqn:collared}) from Theorem~\ref{thm:neg-end}, we find that 
$$2r \leq w\left((\rr \times \leg_+)_{-\infty}^0 \right)\leq w\left(  L_{-\infty}^0 \right).$$
By Theorem~\ref{thm:ub} and the fact that  $\mc_{L}(\leg_{+})$ and $\Mc(\leg_{+})$ are always the height of a Reeb chord, we find
$$w\left(  L_{-\infty}^0 \right) \leq 2\mc_{L}(\leg_{+}) \leq 2\Mc(\leg_{+}) \leq 2r. $$
\end{proof}

 \begin{ex} \label{ex:sharp} Corollary~\ref{cor:calc} immediately implies the following calculations.

\begin{enumerate}
\item  \label{item:unknots} If $U^n(r)$ denotes the $n$-dimensional Legendrian unknot with precisely one Reeb chord of height $r$, then
$$ w\left( \left( \rr \times U^n(r) \right)_{-\infty}^0 \right) = 2 r.$$

\item If $T^{1}(r)$ denotes the $1$-dimensional Legendrian trefoil shown in Figure~\ref{fig:ex:trefoil}, then
\[ w\left( \left( \rr \times T^1(r) \right)_{-\infty}^0 \right) = 2 r.\]

\item \label{item:isotopic}
More generally, let $\leg \subset J^{1}\rr^{n}$ be a connected Legendrian that admits an augmentation.   Suppose that the front of $\leg$ is contained in a box of height $s$,
and let $\leg^{\#}(r)$ be the Legendrian submanifold
constructed as a cusp connect sum of $\leg$ and $U^{n}(r)$, with $r > s$, as shown in Figure~\ref{fig:connect-sum}; this
construction appears, for example,  in \cite{bst:construct}, \cite{rizell:surgery}, and \cite[\S4]{ees:high-d-geometry}. As in the examples above, we obtain:
\[ w\left( \left( \rr \times \leg^{\#}(r) \right)_{-\infty}^0 \right) = 2 r.\]

\item If $L$ is a fundamental cobordism that is 
cylindrical outside  $[s_-, -\varepsilon]$ and has positive end equal to $U^{n}(r)$, $T^{1}(r)$, or $\leg^\#(r)$,
 then
 $$w\left( L_{-\infty}^0 \right) = 2r.$$
\end{enumerate}
 \end{ex}

 \begin{figure}
	\labellist
	\small
	\pinlabel $r$ [r] at 210 62
	\endlabellist

\centerline{\includegraphics{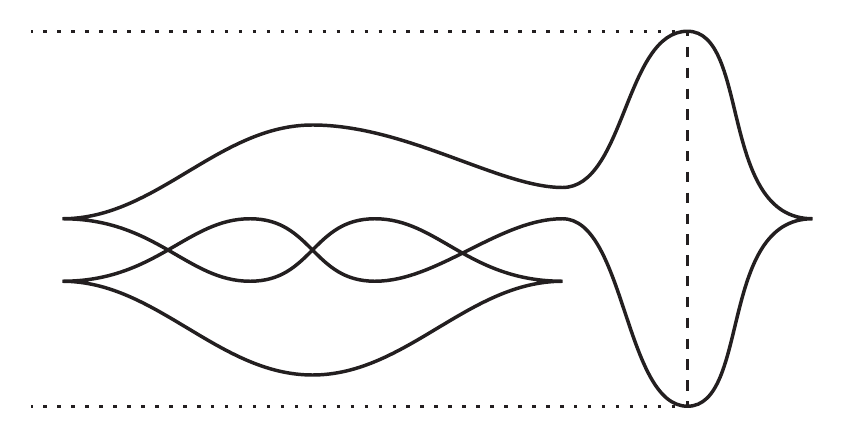}}
\caption{The Legendrian trefoil $T^{1}(r)$.   }
\label{fig:ex:trefoil}
\end{figure}

  \begin{figure}

	\labellist
	\small
	\pinlabel $s$ [l] at -4 70
	\pinlabel $r$ [r] at 290 95
	\endlabellist

\centerline{\includegraphics[width=4in]{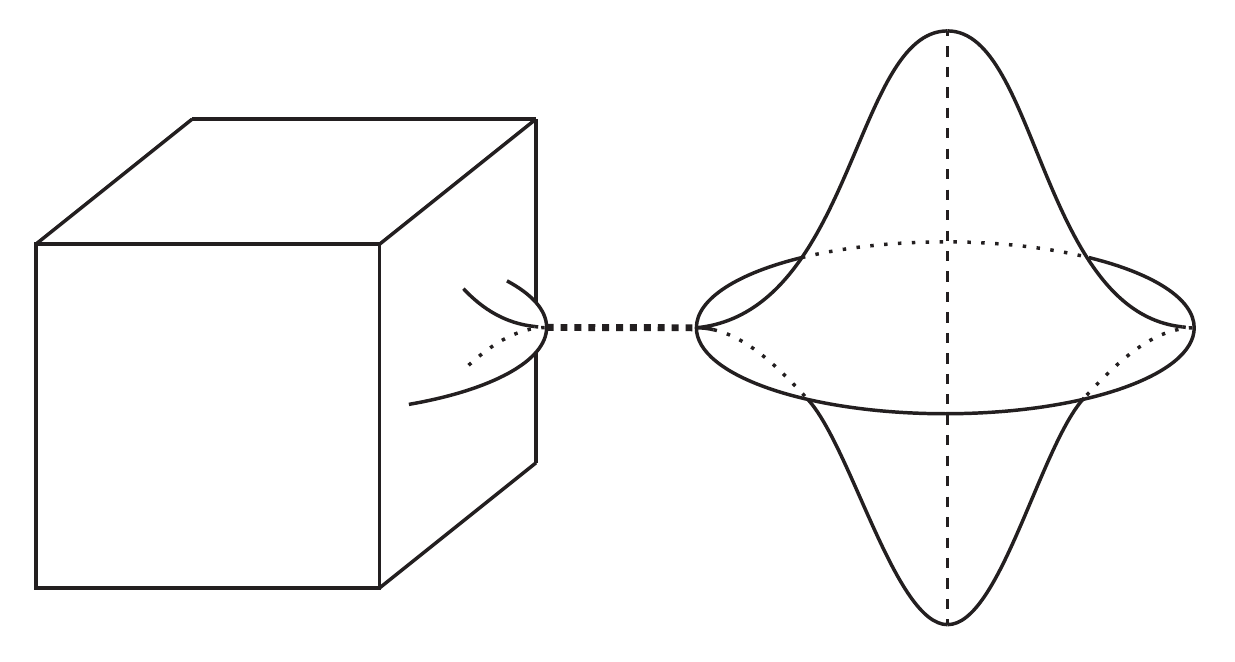}}
\caption{The setup for forming the connect sum between an $n$-dimensional Legendrian $\leg$ at left and a standard $n$-dimensional unknot $U^{n}(r)$ at right.}  
\label{fig:connect-sum}
 \end{figure}

\begin{rem}
Item~(\ref{item:isotopic}) in Example~\ref{ex:sharp} shows that any connected Legendrian $\leg \subset J^1\rr^n$ that admits an augmentation is Legendrian isotopic a Legendrian  where the upper and lower bounds given by Theorem~\ref{thm:ub}  and Theorem~\ref{thm:lower-isolated-chord} are sharp.
  Calculating $w\left( \left( \rr \times \leg \right)_{-\infty}^0 \right)$ in general is more challenging.  In particular, if none of the long Reeb chords of a Legendrian $\leg$ are frontwise doubly extendable, then there will be gaps between the upper and lower bounds that we construct in this paper.
   \end{rem}

 \subsection{Connections to the Length of a Cobordism}
\label{ssec:length-intro}

We  finish our investigations by connecting the relative Gromov width to the length between Legendrians  studied in \cite{josh-lisa:cob-length}.  Given Legendrians $\leg_{\pm}$, the {\bf Lagrangian cobordism length} $\ell(\leg_{-}, \leg_{+})$ is defined to be:
\begin{align*}
\ell(\leg_{-}, \leg_{+}) = \inf \{ s_+ - s_- : &\,  \exists \text{ Lagrangian cobordism $L$ from $\leg_{-}$ to $\leg_{+}$}\\
&\text{that is cylindrical outside $[s_{-}, s_{+}]$} \}.
\end{align*}
If there does not exist a Lagrangian cobordism from $\leg_{-}$ to $\leg_{+}$, then we define $\ell(\leg_{-}, \leg_{+}) := +\infty$.   Recall that the Lagrangian
cobordisms under consideration are exact and satisfy that other conditions of Definition~\ref{defn:cobordism}.
One of the key observations of \cite{josh-lisa:cob-length} was that the cobordism length exhibits flexibility (resp. rigidity) when $\leg_+$ is, in a sense,  larger (resp. smaller) than $\leg_-$. 
 The final main result of this paper is that the length of a  fundamental cobordism is bounded below by a ratio of relative widths.

\begin{thm} \label{thm:length-width}  Suppose $L$ is a fundamental cobordism from $\leg_{-}$ to $\leg_{+}$ that is cylindrical outside $[-s, 0]$.
Then:
\begin{equation} \label{eqn:length-gen}
\ln \left( \frac{  w\left( (\rr \times \leg_-)_{-\infty}^0 \right)}{2 \Mc(\leg_+)} \right) \leq s. 
\end{equation}
If, in addition,  the upper bound to $w\left( (\rr \times \leg_+)_{-\infty}^{0} \right)$ given by  Theorem~\ref{thm:ub} is realized, then
\begin{equation} \label{eqn:length-width}
\ln \left( \frac{  w\left( (\rr \times \leg_-)_{-\infty}^0 \right)   }{w\left( (\rr \times \leg_+)_{-\infty}^{0} \right)   } \right) \leq s.
\end{equation}
\end{thm}

\begin{proof}  Suppose that $L$ is a fundamental Lagrangian cobordism from $\leg_-$ to $\leg_+$ that is cylindrical outside $[-s, 0]$.
Then Lemma~\ref{lem:heights}, the fact that $(\rr \times \leg_-)_{-\infty}^{-s} \subset L_{-\infty}^0$, and Theorem~\ref{thm:ub} imply
$$  e^{-s}w(   (\rr \times \leg_-)_{-\infty}^0  )  =  w((\rr \times \leg_-)_{-\infty}^{-s})  \leq w(L_{-\infty}^0) \leq 2\mc_L(\leg_+) \leq 2\Mc(\leg_+),$$
and the result follows.
\end{proof}  

We can apply this to get lower bounds to cobordism lengths between particular Legendrians.  The corollary below follows immediately from Example~\ref{ex:sharp}(1) and Theorem~\ref{thm:length-width}.
 
\begin{cor}   \label{cor:lengths}
\begin{enumerate}
\item  If $U^n(r_\pm) \subset J^{1}\rr^{n}$ is the $n$-dimensional Legendrian unknots described in Example~\ref{ex:sharp}(\ref{item:unknots}), then
$$\ln \left( \frac{ r_- }{ r_+   } \right) \leq \ell\left( U^n(r_-),   U^n(r_+) \right).$$
\item Given $\leg_{\pm} \subset J^{1}\rr^{n}$, where $\leg_{+}$ is connected and each component of $\leg_{-}$ admits an augmentation,
construct $\leg^{\#}_{\pm}(r_{\pm})$ from $\leg_{\pm}$ as in Example~\ref{ex:sharp}(\ref{item:isotopic}).  The bound above  generalizes to:
$$\ln \left( \frac{ r_- }{ r_+   } \right) \leq \ell\left( \leg_{-}^{\#}(r_-),   \leg_{+}^{\#}(r_+) \right).$$
\end{enumerate}
\end{cor}

\begin{rem}
Statement (1) appears as the upper bound in  \cite[Theorem 1.1]{josh-lisa:cob-length}; Statement (2) strengthens the lower bound given in
 \cite[Proposition 6.1]{josh-lisa:cob-length}.
\end{rem}

\subsection{Outline of the Paper}

The remainder of the paper is organized as follows.  After recalling background notions and setting precise definitions in Section~\ref{sec:background}, we set down some basic facts about relative Gromov widths of cobordisms in Section~\ref{sec:first-results}; this section includes a proof of Theorem~\ref{thm:neg-end}. We describe the constructions necessary to prove Theorem~\ref{thm:lower-isolated-chord} in Section~
\ref{sec:lower}.  Section~\ref{sec:j-curves} provides the background necessary to understand the $J$-holomorphic curves used in the proof of Theorem~
\ref{thm:ub}, while Section~\ref{sec:capacity} contains the definition of the fundamental capacity.  Knowledge of 
those $J$-holomorphic curves and the fundamental capacity is put to use in Section~\ref{sec:upper}, where we prove Theorem~\ref{thm:ub}.

\subsection*{Acknowledgements}

We thank Georgios Dimitroglou Rizell, Michael Sullivan, and Egor Shelukhin for stimulating conversations.  We are extremely appreciative to the referee for pointing out some
errors in earlier versions of this paper.  The authors gratefully acknowledge the hospitality of the Institute for Advanced Study during portions of the preparation of this paper.

\section{Background Notions}
\label{sec:background}

In this section, we recall the definitions of our main objects of study: Legendrian submanifolds and Lagrangian cobordisms between Legendrian submanifolds.  We assume basic familiarity with these notions; see, for example, Etnyre's survey on Legendrian submanifolds \cite{etnyre:knot-intro} and Audin, Lalonde, and Polterovich \cite{audin-lalonde-polterovich} on Lagrangian submanifolds.

\subsection{Legendrian Submanifolds}

A \dfn{Legendrian submanifold} $\leg$ of a contact $(2n+1)$-manifold $(Y, \ker \alpha)$ is an $n$-dimensional submanifold whose tangent spaces lie in the contact hyperplanes $\ker \alpha$.  A \dfn{Reeb chord} of $\leg$ is an integral curve of the Reeb vector field of $\alpha$ whose endpoints both lie on $\leg$.  Let the collection of Reeb chords of $\leg$ be denoted by $\mathcal{R}_\leg$.  The \dfn{height} of a Reeb chord $\gamma$ is simply
\begin{equation} \label{eqn:height}
h(\gamma) = \int_\gamma \alpha. 
\end{equation}

We work with closed Legendrian submanifolds in the contactization of a Liouville manifold.
Specifically, let $P$ be a {\bf Liouville manifold}:  $P$ is an open, even-dimensional manifold with a $1$-form $\lambda$ such that $d\lambda$ is symplectic.    We 
construct the {\bf contactization} $\cont{P} = P \times \rr$ with the contact form  $\alpha = dz - \lambda$.  The Reeb flow is parallel to the $\rr$ coordinate of $\cont{P}$, and Reeb chords are in bijective correspondence with double points of the projection of $\leg$ to $P$. A Legendrian submanifold is \dfn{chord generic} if those double points are transverse.

A special case of the contactization of a Liouville manifold is the $1$-jet space of a smooth manifold $M$, namely $J^1M = T^*M \times \rr$ with the canonical contact form $dz-\lambda_{can}$.  We denote the projections to the base and to the $z$ direction by $\pi_x: J^1M \to M$ and $\pi_z: J^1M \to \rr$, respectively.  The {\bf front projection} is the projection   $\pi_{xz}: J^1M \to M \times \rr$.

\subsection{Lagrangian Cobordisms} \label{ssec:cobord}
The Lagrangians we consider
live in the {\bf symplectization}  $\left(\rr \times \cont{P}, d(e^s \alpha)\right)$. 
 \begin{defn} \label{defn:cobordism}
	Given closed Legendrians $\leg_{\pm} \subset \cont{P}$, a \dfn{Lagrangian cobordism  from $\leg_{-}$ to $\leg_{+}$} is an orientable, {properly embedded,} Maslov zero, exact Lagrangian submanifold $L \subset (\rr \times \cont{P}, d(e^s\alpha))$ such that there exist  real numbers $s_- \leq s_+$ satisfying:
	\begin{enumerate}
	\item $L \cap \left( (-\infty, s_-] \times \cont{P} \right) = (-\infty, s_-] \times \leg_-$,
	\item $L \cap \left([s_+,\infty) \times \cont{P}\right) = [s_+, \infty) \times \leg_+$, and
	\item there exist constants $C_\pm$ and a function $f$ such that $e^s \alpha|_L = df$ and on $(-\infty, s_-] \times \leg_-$, $f=C_- = 0$.
		\end{enumerate}
	We say that $L$ is \dfn{cylindrical} outside of $[s_-,s_+]$. 
	\end{defn}

\section{First Results about the Width of Cobordisms}
\label{sec:first-results}

In this section, we collect some foundational results about the relative Gromov width of a Lagrangian cobordism.  
Theorem~\ref{thm:neg-end} is a special case of Theorem~\ref{thm:collar}.

  The first lemma tells us that the relative Gromov width is non-zero.

\begin{lem}\label{lem:small-ball}  For any Lagrangian cobordism $L \subset \rr \times \cont{P}$ and any $p \in L$, there exists $r > 0$ and a relative symplectic embedding
$\psi: B^{2n}(r) \hookrightarrow (\rr \times \cont{P}, L)$  with  $\psi(0) =p$. 
\end{lem}

\begin{proof} The lemma essentially follows from Weinstein's Lagrangian neighborhood theorem, suitably adjusted to non-compact Lagrangian cobordisms.
\end{proof}

The next lemma explains why we focus our study of relative Gromov width to the lower halves of Lagrangian cobordisms.

 \begin{lem}  \label{lem:infinite-top} For any Lagrangian cobordism $L \subset \rr \times \cont{P}$ and any $-\infty \leq s_0 < \infty$, we have 
 $w(\lag{s_0}{\infty}) = \infty$.
\end{lem}

\begin{proof}   Suppose that $L$ is a Lagrangian cobordism from $\leg_-$ to $\leg_+$ that is cylindrical outside $[s_-, s_+]$.  Choose
$p \in L \cap (s_+,\infty) \times \cont{P}$.    By Lemma~\ref{lem:small-ball}, there exists a $r > 0$ and a relative symplectic embedding 
\[\epsilon_r:  B^{2n}(r)  \hookrightarrow \left((s_+, \infty) \times \cont{P}, (s_+, \infty) \times \leg_+ \right).\]
Fix an arbitrary $R > r$.  We will construct a relative symplectic embedding 
\[\psi: B^{2n}(R) \hookrightarrow \left((s_+, \infty) \times \cont{P}, (s_+, \infty) \times \leg_+\right)\] by precomposing and postcomposing
$\epsilon_r$ with maps that are conformally symplectic. 

To define the first map, let $\lambda = r/R$ and consider the scaling map $\kappa_\lambda: B^{2n}(R) \to B^{2n}(r)$ given by $\kappa_\lambda (\mathbf{x},\mathbf{y}) = \sqrt{\lambda}(\mathbf{x},\mathbf{y})$. Observe that $\kappa_\lambda^* \omega_0 = \lambda \omega_0$, and that $\kappa_\lambda$ preserves the Lagrangian plane $\rr^n \subset \rr^{2n}$.

For the second map, let $u =  - \ln \lambda > 0$ and  consider the translation by $u$ of the positive end of the cobordism: 
$$\begin{aligned}
\tau_u: (s_+ , \infty) \times \cont{P} &\to (s_+ + u, \infty) \times \cont{P} \\
 (s, p) &\mapsto  (s + u, p).
\end{aligned}$$
Observe that $\tau_u^* \omega = e^u \omega =  \frac{1}{\lambda} \omega$.

Putting the maps above together yields $\psi = \tau_{-\ln \lambda} \circ \epsilon_r \circ \kappa_{\lambda}$, the desired relative symplectic embedding from $B^{2n}(R)$ to $((s_+, \infty) \times \cont{P}, (s_+, \infty) \times \leg_+)$.
\end{proof}

\begin{rem}
	A similar style of argument appears in Dimitroglou Rizell's proof that if a closed Lagrangian in a symplectization has a neighborhood that is equal to a cylinder over a Legendrian, then it has infinite relative Gromov width \cite{rizell:uniruled}. Dimitroglou Rizell's argument needed to be more complicated since his Lagrangians were not cylindrical at infinity.
\end{rem}

As a result of Lemma~\ref{lem:infinite-top},  we will restrict our attention to the relative widths of negative ends of Lagrangian cobordisms, i.e. Lagrangians of the form $L_{-\infty}^b$,
where $s_+ \leq b$. 

For the special case where $L$ is cylindrical over a Legendrian, the following lemma shows that it suffices to understand the width of $L_{-\infty}^0$. The proof is analogous to that of Lemma~\ref{lem:infinite-top}.

 \begin{lem} \label{lem:heights} For any Legendrian $\leg \subset \cont{P}$, 
  $$w((\rr \times \leg)_{-\infty}^b) = e^b w((\rr \times \leg)_{-\infty}^0).$$ 
  \end{lem}

The next argument shows that regions of a Lagrangian cobordism that  are cylindrical over 
a Legendrian $\leg$ can be stretched downward,  which
allows us to get a lower bound for the relative width of a Lagrangian cobordism from the width of
$(\rr \times \leg)_\infty^0$.

\begin{thm} \label{thm:collar} Suppose $L \subset \rr \times \PR$ is a Lagrangian cobordism that is cylindrical
outside $[s_-, 0]$.  Suppose for $a < b \leq 0$, there exists a  Legendrian $\leg \subset \PR$ 
so that 
$$L \cap \left([a,b] \times \cont{P}\right) =  [a, b] \times \leg.$$
Then
$$w\left(  (\rr \times \leg)_{-\infty}^{b} \right) \leq w\left(L_{-\infty}^0\right).$$
\end{thm}

\begin{proof} We will show that if there exists a relative symplectic  embedding 
$\psi_0: B^{2n}(r) \hookrightarrow ((-\infty, b) \times \cont{P},(-\infty, b) \times \leg)$, 
then there also exists
a relative symplectic  embedding $\psi_1: B^{2n}(r) \hookrightarrow ((-\infty, 0) \times \cont{P},L_{-\infty}^0)$.
 
  Given the relative symplectic embedding $\psi_0$, 
suppose that $\im \psi_0 \subset (k, b) \times \PR$ for some $k < b$.    If $a \leq k$, then we can take $\psi_1 = \psi_0$.
If, on the other hand, we have $k < a$, we stretch the cylinder as follows.  Fix  constants $u,v$ so that $v < u <  a < b$, and
let $\rho(s): \rr \to \rr$ be a smooth, non-positive, compactly supported function with 
\begin{equation*}
\rho(s) = \begin{cases}
	0 & s \leq v \\
	k-a & s \in [u,a] \\
	0 & s \geq b
\end{cases}
\end{equation*}
 By an appropriate choice of $u,v$, we can guarantee that
$\rho'(s)>-1$, which guarantees that for all $t \in [0,1]$,  $\sigma_t(s) := s + \rho(s)t$ is an injective function. 

Next, we consider the isotopy of $L$ given by
\[\begin{aligned}
\phi: [0,1] \times L   &\to  \rr \times \PR  \\
(t, s, p) &\mapsto (\sigma_t(s), p).
\end{aligned}\]
We write $\phi_t: L \to \rr \times \cont{P}$ for the restriction of $\phi$ to $\{t\} \times L$. It is easy to verify that $L_t = \phi_t(L)$ is a $1$-parameter family of  exact 
 Lagrangian submanifolds with $L_0 = L$ and
$L_1 \cap \left( (k,b) \times \cont{P}\right) = (k, b) \times \leg$.  
It is a well known fact that exact Lagrangian isotopies of compact manifolds can be realized by Hamiltonian isotopies; see, for example,
\cite[\S 2.3]{audin-lalonde-polterovich} or \cite[\S 3.6]{oh:sym-top+floer-hom}.   Even though $L$ is not compact,
$\phi_t$ is a compactly supported, and so the proofs in the compact setting imply that 
there exists a {Hamiltonian} isotopy $h_t$ of $\rr \times \PR$ such that $h_t(L) = L_t$ and $h_t = \operatorname{id}$ when  $s \leq v$ or $s \geq b$.
 The map $\psi_1 = h_1^{-1} \circ \psi_0$ is our desired relative symplectic embedding.
\end{proof}

Observe that Theorem~\ref{thm:neg-end} follows immediately from Lemma~\ref{lem:heights} and Theorem~\ref{thm:collar}.

\section{Constructing Embeddings near Extendable Reeb Chords}
\label{sec:lower}

The goal of this section is to prove the lower bound to the relative Gromov width given in Theorem~\ref{thm:lower-isolated-chord} by
constructing  relative symplectic embeddings using a neighborhood of a suitably extendable Reeb chord.
    The Reeb chords we are interested in are characterized as follows.

  \begin{defn} \label{defn:extendable} Suppose $\gamma$ is a  Reeb chord   of $\leg \subset J^1M$ whose endpoints are disjoint from the preimages of all singularities of the front projection.   Let
   $\pi_x(\gamma) = x_0$ and $\pi_z(\gamma) = [z_-, z_+]$.  If $h=z_+ - z_-$, let $Z_+$ and $Z_-$ denote the
   forward and backward extensions of  $\pi_z(\gamma)$ to intervals of height $2h$:
   $$Z_- = [z_- - h, z_+], \qquad Z_+ = [z_-, z_+ + h].$$
   Then $\gamma$ is 
   \dfn{frontwise doubly extendable} if either 
   $$\pi_{xz}^{-1}(\{x_0\} \times Z_- ) \cap \leg = \leg \cap \partial \gamma \quad \text {or } \quad
   \pi_{xz}^{-1}(\{x_0\} \times Z_+) \cap \leg = \leg \cap \partial \gamma.$$
  \end{defn}

We are now ready to prove Theorem~\ref{thm:lower-isolated-chord}. 
  \begin{proof}[Proof of Theorem~\ref{thm:lower-isolated-chord}]  Let $\leg \subset J^1M^{n}$ be a Legendrian submanifold  with
   a frontwise doubly extendable Reeb chord $\gamma$ of height $h$ and endpoint heights of $z_\pm$.  We want to prove that
  \[2h \leq w\left((\rr \times \leg)_{-\infty}^0\right).\]
  For any  $r$ such that $\pi r^2 < 2h$, we will  construct  a relative symplectic
  embedding  $\tau: B^{2n+2}(r) \hookrightarrow ((\rr \times J^1M)_{-\infty}^0,(\rr \times \leg)_{-\infty}^0)$.  The construction of the embedding will proceed in three steps:  after setting notation more carefully, we will change the target manifold from the symplectization to the (symplectomorphic) cotangent bundle $T^*(\rp \times M)$, where certain computations are easier.  Next, we will embed a polydisk into a subset of $T^*(\rp \times M)$, sending the real part of the polydisk to the zero section.  Finally, we will adjust the embedding of the polydisk to match the Lagrangian $\rr \times \leg$ along the real part; the desired embedding of a ball follows from restricting the embedding from the polydisk to the round ball.
   
   Before beginning the key steps in the proof, let us set notation more precisely. Suppose $\gamma$ is a frontwise doubly extendable Reeb chord of height $h$ in the downward direction, i.e. we are using the interval $Z_-$ in Definition~\ref{defn:extendable}; the proof for the upward direction is analogous.  We begin by refining the neighborhood of the extended Reeb chord. {By our assumption that the endpoints of the Reeb chord are disjoint from preimages of singularities of the front projection, } there is a neighborhood $V$ of $x_0$ in $M$ such that in  $U = \pi_{xz}^{-1}(V \times Z_-)$,
   $\leg$ is a disjoint union of the
 $1$-jets of two functions $f_\pm: V \to \rr$ with $f_+ > f_-$ on $V$: 
 \begin{equation*}
 \leg \cap U= j^1(f_+) \cup j^1(f_-).
 \end{equation*}
 Fix an arbitrary $\varepsilon_1$ satisfying $0 < \varepsilon_1 < h/2$.
 By shrinking the neighborhood $V$ of $x_0$, we can assume that for all $x \in V$, we have
 \begin{equation} \label{eqn:space}
  f_+(x) - f_-(x) >  h - \varepsilon_1 \quad \text{ and } \quad
   |f_\pm(x) - z_\pm | <  \varepsilon_1/2.  
 \end{equation}
 We will restrict attention in the target of the embedding to the symplectization of $U$ relative to the symplectization of the image of $J^1(f_-)$. 

The first step in the proof is to transform the target of the embedding into a subset of a cotangent bundle. Consider 
the symplectic diffeomorphism 
\begin{equation*}
\begin{aligned}
\Psi: \rr \times J^1M &\to  T^*\rp \times T^*M\\
(s, x, y, z) &\mapsto  (e^s, z, x, e^sy).
\end{aligned}
\end{equation*}

We may parameterize the image of the cylinder over $\leg \cap U$ in $T^{*}(\rp \times M)$ as follows:
\begin{align}
	\Psi((-\infty,0) \times j^1(f_-)) &= \{(t,f_-(x),x,t\, df_-(x))\}, \label{eq:param-j1} \\
	\Psi((-\infty,0) \times j^1(f_+)) &= \{(t,f_+(x),x,t\, df_+(x))\}.
	\end{align}
Here  $t \in (0,1)$ and $x \in V$.

Our goal now is to construct a symplectic embedding of $B^{2n+2}(r)$ into $\Psi(  (-\infty, 0) \times U  ) \subset T^* \rp \times T^*M$ relative to $\Psi((-\infty, 0) \times j^1(f_-))$.  To make the capacity $\pi r^2$ more concrete, we fix an arbitrary $\varepsilon_2$ satisfying $0 <  \varepsilon_2 < 1$, and let 
$\pi r^2 = 2(h-\varepsilon_1)(1 -\varepsilon_2)$.  Since the $\varepsilon_i$ are arbitrary, we see that the supremum of the capacities of the embeddings constructed here is, indeed, $2h$.

The second major step is to construct a symplectic embedding
 $\sigma: B^2(r) \times B^{2n}(r) \hookrightarrow T^*\rp \times T^*M$ that sends the $(B^{2}(r) \cap \rr) \times (B^{2n}(r) \cap \rr^{n})$ to the
 zero section of the cotangent bundle.  To make the construction more precise, construct sets $A \subset \rp$, $B \subset \rr$, and
$C \subset V \subset M$.  We want the embedding $\sigma$ to send $B^2(r) \times B^{2n}(r)$ into $(A \times B) \times T^*C$. To define the first component of $\sigma$, let $A = [\frac{\varepsilon_2}2, 1-\frac{\varepsilon_2}{2}]$ and $B = [-h+ \varepsilon_1, h - \varepsilon_1]$; 
observe that the area of the rectangle $A \times B$ is $\pi r^2$, and thus there exists a relative symplectic embedding
$$\sigma_1 : B^2(r) \hookrightarrow (A \times B, A \times \{0\}).$$ 
 Next, 
for the domain $V \subset M$ of $f_\pm$,  choose a non-empty closed set $C \subset V $ that is diffeomorphic to $B^n$ and contains a neighborhood of $x_0$.  Thus there exists a symplectic diffeomorphism between $T^*B^n$ and $T^*C$, and so
 there exists a relative symplectic embedding
  $$\sigma_2 : B^{2n}(r)\hookrightarrow  ( T^*C, C_0),$$ 
 where $C_0$ denotes the zero section of $T^*C$. Putting the foregoing construction together, we see that $\sigma = \sigma_1 \times \sigma_2$ restricts to a
 define a relative symplectic embedding of $\sigma: B^{2n+2}(r) \to ((A \times B) \times  T^*C,  (A \times \{0\}) \times  C_0)$.
 
The final step is to adjust the embedding $\sigma$ so that its real part lies in  $\Psi((-\infty, 0) \times j^1(f_-))$ rather than the zero section of $T^*\rp \times T^*C$.  Let  $W =T^*\rp \times T^*V$, and consider the symplectic embedding 
\begin{align*}
 \phi: (W, \omega_0) &\to (T^*\rp \times T^*M, \omega_0) \\
\phi(t,u,x,y) &= (t, u + f_-(x), x, y + t\,df_-(x)).
\end{align*}

We finish the proof by defining the desired relative symplectic embedding by $\tau = \phi \circ \sigma$ and verifying that it has the necessary properties. By construction, we have $\phi(A \times B \times T^*C ) \subset \Psi( (-\infty, 0) \times U)$, 
and thus
 $\tau(B^{2n+2}(r)) \subset    \Psi( (-\infty, 0) \times U)$.
A straightforward verification shows that the image of the real part of $B^{2n+2}(r)$ under $\tau$ is parametrized by points of the form $(a,f_-(b),b,a\, df_-(b))$, which certainly lies within $\Psi(j^1(f_-))$ by Equation~(\ref{eq:param-j1}).
To show that no other points of $\tau(B^{2n+2}(r))$ are contained in
$\Psi( (-\infty, 0) \times (U \cap \leg))$,
it suffices to show that there is no point in $\phi(A \times B \times T^*C )$ of the form
$(a, f_+(c), c,  a\,df_+(c))$.   
Equation~(\ref{eqn:space}) implies that, for any $b \in B$ and $c \in V$, we have $b < h-\varepsilon_1$ and $f_+(c) - f_-(c) > h - \varepsilon_1$.  We then see that   
$b + f_-(c) < h-\varepsilon_1 + f_-(c) < f_+(c).$
Thus $\phi(a, b, c, d)$ cannot be of the form $(a, f_+(c), c,  a\,df_+(c))$, as desired.    
\end{proof}

\section{$J$-Holomorphic Curves and the Fundamental Class}
\label{sec:j-curves}

In this section, we lay out the algebraic and analytic structures that we will use to derive the upper bound in Theorem~\ref{thm:ub}.  Recall that the strategy for the proof of Theorem~\ref{thm:ub} is to find a $J$-holomorphic curve of controlled area through the center of the image of a given symplectic embedding of a ball.   In this section, we will prove the existence of appropriate $J$-holomorphic curves in Corollary~\ref{cor:uniruling}.

In Sections~\ref{ssec:j-disk} and \ref{ssec:lch}, 
we will recall the constructions underlying  Floer theory for Lagrangian cobordisms from \cite{c-dr-g-g-cobordism} and Legendrian Contact Homology (LCH) from \cite{ees:pxr}.  In Section~\ref{ssec:fundamental-class}, we outline the construction of the fundamental class  following \cite{c-dr-g-g-cobordism} and examine how the fundamental class implies the existence of
$J$-holomorphic disks that will be used in Sections~\ref{sec:capacity} and \ref{sec:upper}.

\subsection{Moduli Spaces of $J$-Holomorphic Disks}
\label{ssec:j-disk}

In this section, we define the moduli spaces of $J$-holomorphic disks that will underlie the algebraic structures defined in later sections.  We follow the language of \cite[\S3]{c-dr-g-g-cobordism}. 
The geometric background for the moduli spaces begins with two chord-generic links $\leg_{-} \cup \tleg_{-}$  and $\leg_{+} \cup \tleg_{+}$ in $\cont{P}$.  Next, consider a pair of Lagrangian cobordisms $L, \tL$ in the symplectization $\rr \times \cont{P}$  from $\leg_{-}$ to $\leg_{+}$ 
(resp. from  $\tleg_{-}$ to $\tleg_{+}$); see Definition~\ref{defn:cobordism} for the hypotheses satisfied by these Lagrangian cobordisms. 

To define the $J$-holomorphic disks themselves, we let $D_k$ denote the closed unit disk in $\cc$ with $k+1$ punctures $z_0, \ldots, z_k$ on its boundary.  A disk $D_k$ will come with a distinguished puncture $z_j$ for $j>0$, which splits $\partial D_k$ into two parts: one from $z_0$ to $z_j$ (counterclockwise) called $\partial_- D_k$, and one from $z_j$ to $z_0$ called $\partial_+ D_k$. 
Each of  $\partial_\pm D_k$ will be labeled with a Lagrangian cobordism $\underline{L}(\pm)$ as in \cite[{\S3.2.1}]{c-dr-g-g-cobordism}.

When defining moduli spaces of $J$-holomorphic disks, we use compatible almost complex structures on the symplectization $\rr \times \cont{P}$ satisfying
different conditions, depending on whether or
not the Lagrangian cobordism is cylindrical and whether or not the boundary
of the disk lies on a single Lagrangian or ``jumps'' between different Lagrangians.  
We set notation for these complex structures here and refer the reader to \cite[{\S2.2}]{c-dr-g-g-cobordism} for further details.

 The spaces of complex structures we need are the following:
\begin{itemize}
\item Almost complex structures that are cylindrical over the entire symplectization  will be denoted $\mathcal{J}^{cyl}$.  In 
particular, these almost complex structures are invariant under the conformal 
action of $\rr$ on $\rr \times \cont{P}$. Further, such an almost complex structure maps the {symplectization}  direction to the {Reeb} direction, preserves the contact planes, and is compatible with $d\lambda|_\xi$ on the contact planes. The {sub}set of cylindrical almost complex structures that come from lifting an admissible almost complex structure on $P$ is denoted $\mathcal{J}^{cyl}_\pi$.  This type of almost complex structure will be used in defining ``pure LCH moduli spaces'' in Definition~\ref{defn:lch-ms}.

\item Almost complex structures that agree with those in $\mathcal{J}^{cyl}_\pi$ outside of a compact set contained in $[s_-, s_+] \times \cont{P}$ will be denoted $\mathcal{J}^{adm}$.  Domain-dependent almost complex 
structures are defined using a path $J_t \in \mathcal{J}^{adm}$; these almost complex structures fit together to  define a map $\mathbf{J}$ on the Deligne-Mumford space of punctured disks.  These maps $\mathbf{J}$ arise in the construction of the the ``mixed LCH moduli space'' (Definition~\ref{defn:lch-mixed-ms}) and the ``Floer to LCH moduli space'' (Definition~\ref{defn:floer-lch-ms}).
\end{itemize}

Our next goal is to describe moduli spaces of $J$-holomorphic maps $u: D_k \to \rr \times \cont{P}$ with specific behaviors along the boundary and near the punctures as in \cite[\S3]{c-dr-g-g-cobordism}.  We are given  a punctured disk with a Lagrangian label $\underline{L}$ and an almost complex structure $J$ on $\rr \times \cont{P}$, possibly domain dependent. We say that a map $u: D_k \to \rr \times \cont{P}$ is \dfn{$J$-holomorphic with Lagrangian boundary conditions} if it satisfies:
\begin{enumerate}
\item[(J1)] $du \circ j = J \circ du$, and
\item[(J2)] $u(\partial_\pm D_k) \subset \underline{L}(\pm)$.   
\end{enumerate}

The relevant moduli spaces of $J$-holomorphic disks will differ in the asymptotics of their maps near the boundary punctures.  To specify those conditions, we note that a neighborhood of a boundary puncture $z_i$ of $D_k$ is conformally equivalent to a strip $S = (0,\infty) \times i[0,1] \subset \cc$, and we let $(s,t)$ denote the coordinates on $S$. 

The first type of asymptotic condition involves a Reeb chord $\gamma$ of $\leg \cup \tleg$ of length $T$. We say that a $J$-holomorphic map with Lagrangian boundary conditions  $u = (a,v): S \to \rr \times \cont{P}$ is \dfn{$\pm$-asymptotic to $\gamma$ at $z_i$} if it satisfies the following conditions in local coordinates on the strip $S$:
\begin{enumerate}
\item[(R1)] $\lim_{s \to \infty} a(s,t) = \pm \infty$, and
\item[(R2$+$)] For $+$-asymptotic, $\lim_{s \to \infty} v(s,t) = \gamma(Tt)$, or
\item[(R2$-$)] For $-$-asymptotic, $\lim_{s \to \infty} v(s,t) = \gamma(T(1-t))$.  
\end{enumerate}

The second type of asymptotic condition involves an intersection point $m \in L \cap \tL$.  We say that a $J$-holomorphic map with Lagrangian boundary conditions  $u: D_k \to \rr \times \cont{P}$ is \dfn{asymptotic to $m$ at $z_i$} if it satisfies the following condition:
\begin{enumerate}
\item[(I1)] $\lim_{z \to z_i} u(z) = m$.
\end{enumerate}

We can now define a number of moduli spaces that will be used in later arguments.  
The first moduli space is important for a single Legendrian. 

\begin{defn}[\S3.1 in \cite{c-dr-g-g-cobordism}] \label{defn:lch-ms}
	For a Reeb chord $a$ of $\leg$, a word $\mathbf{b} = b_1 \cdots b_k$ of Reeb chords of $\leg$, and a $J \in \mathcal{J}^{cyl}_\pi$, 
	we define the \dfn{pure LCH moduli space} $\ms^J_{\leg}(a;\mathbf{b})$ to be  the set  
	of $J$-holomorphic maps with Lagrangian boundary conditions labeled by $\underline{L} = \rr \times \leg$ that are $+$-asymptotic to $a$ at $z_0$ and are $-$-asymptotic to the other Reeb chords at the corresponding punctures, up to conformal reparametrization of the domain. Note that the moduli space $\ms^J_{\leg}(a;\mathbf{b})$ admits an $\rr$-action via translation in the symplectization direction. 

\end{defn}
A similar moduli space is useful for a pair of Legendrians.

\begin{defn} [\S3.2 in \cite{c-dr-g-g-cobordism}] \label{defn:lch-mixed-ms}
	The \dfn{mixed LCH moduli space} 
	$$\ms^{\mathbf{J}}_{\leg \leftarrow \tleg}\left(\mixed{a};\mathbf{b},\mixed{c}, \tb \right)$$ 
	is defined similarly to the pure LCH moduli space, with $\mathbf{J}$ induced by a path in $\mathcal{J}^{adm}$, $\mixed{a}$ and $\mixed{c}$ Reeb chords from $\tleg$ to $\leg$,  and $\mathbf{b}$ (resp. $\tb$) a word of Reeb chords of $\leg$ (resp. $\tleg$).  The Lagrangian boundary conditions have labels  $\underline{L}(-) = \rr \times \leg$ and $\underline{L}(+) = \rr \times \tleg$.	 
	 \end{defn}

We quickly review the construction of Floer cohomology groups for Lagrangian cobordisms to fix notation. Given exact Lagrangian cobordisms $L$ and $\tL$ that intersect transversally and whose Legendrian ends are disjoint, we may define the \dfn{Floer cochain complex} $FC^*(L, \tL)$ to be generated by the intersection points $L \cap \tL$ over $\ff_2$.  We grade the chain complex using the Conley-Zehnder index as in \cite[\S3.3]{c-dr-g-g-cobordism}, and we define the differential $d_{00}:FC^*(L, \tL) \to FC^{*+1}(L, \tL)$ on an intersection point $x \in L \cap \tL$ by {a count of
appropriate $J$-holomorphic curves that are schematically shown in \cite[Figure 3]{c-dr-g-g-cobordism}; see also \cite[\S 3.2.3]{c-dr-g-g-cobordism}.
 The cohomology   of this complex is the \dfn{Floer cohomology} $FH^*(L, \tL)$.} 

The next moduli space we define deals with pairs of Lagrangian cobordisms.
 
\begin{defn}[\S3.2.5 in \cite{c-dr-g-g-cobordism}] \label{defn:floer-lch-ms} Suppose    $L$ (resp. $\tL$) is a 
Lagrangian cobordism from $\leg_{-}$ to $\leg_{+}$ (resp. $\tleg_{-}$ to $\tleg_{+}$),  
 $\mixed{a}$ is a Reeb chord from $\tleg_{+}$ to $\leg_+$, $m \in L \cap \tL$,
 $\mathbf{b} = b_{1} \cdots b_{j-1}$ is a word
of Reeb chords of $\leg_{-}$,  
and $\tb = \tilde{b}_{j+1} \cdots \tilde b_{k}$ 
is a word of Reeb chords of $\tleg_{-}$.   
Given a path  $\mathbf{J} \in \mathcal{J}^{adm}$, we define the \dfn{Floer to LCH moduli space} 
$$\ms^{\mathbf J}_{L \leftarrow \tL} (\mixed{a};\mathbf{b}, m,  \tb)$$ to be the set of 
	$\mathbf{J}$-holomorphic maps with Lagrangian boundary  labels  $\underline{L}(-) = L$ and $\underline{L}(+) = \tL$
	 that are $+$-asymptotic to $a$ at $z_0$ and are $-$-asymptotic $b_{i}$ at $z_{i}$, for $i \in \{1, \dots,  j-1\}$, to $m$ at $z_{j}$,
	 and to $\tilde{b}_{i}$ at $z_{i}$, for $i \in \{ j+1, \dots, k\}$.  A schematic representation of a curve in 
	 $\ms^{\mathbf J}_{L \leftarrow \tL} (\mixed{a};\mathbf{b}, m,  \tb )$ can be found in \cite[Figure 5]{c-dr-g-g-cobordism}. 
\end{defn}	
 
As shown by Proposition 3.2 in \cite{c-dr-g-g-cobordism}, among other sources, all of these moduli spaces are transversally cut out, pre-compact manifolds for generic $J$  or $\mathbf{J}$, and hence are finite sets when their dimension is $0$.

\subsection{The Chekanov-Eliashberg DGA and its Linearizations}
\label{ssec:lch}

To define the differential graded algebra (DGA) for a Legendrian submanifold $\leg$ underlying Legendrian Contact Homology (LCH), we begin with the $\ff_2$-vector space $A_\leg$ generated by the set of Reeb chords $\mathcal{R}_\leg$ as in \cite{ees:high-d-analysis, ees:high-d-geometry, ees:pxr}.
We then define $\alg_\leg$ to be the unital tensor algebra $TA_\leg = \bigoplus_{i=0}^\infty A_\leg^{\otimes i}$.  The generators of $A_\leg$ are graded by a Conley-Zehnder index, with the grading extended to $\alg_\leg$ additively.  The gradings are well-defined up to the Maslov number of the Lagrangian projection of $\leg$ to $P$.

The differential of a Reeb chord $a \in \mathcal{R}_\leg$ counts $0$-dimensional moduli spaces from Definition~ \ref{defn:lch-ms}:
\begin{equation} \label{eq:diffl}
	\df_\leg (a) = \sum_{\dim \ms^J_{\leg}(a;\mathbf{b}) = 1} \#(\ms^J_{\leg}(a;\mathbf{b})/\rr)\, \mathbf{b}.
\end{equation}
The count is taken modulo $2$. The differential extends to all of $\alg_\leg$ via linearity and the Leibniz rule and has degree $-1$.  Compactness and gluing arguments show that $(\df_\leg)^2=0$ as in \cite{ees:high-d-analysis, ees:pxr}.

It is notoriously difficult to  extract computable invariants from the full  DGA.  One way to render the theory more computable is to use Chekanov's linearization technique \cite{chv}.  Though it is not always possible to use this technique \cite{fuchs:augmentations, fuchs-ishk, rulings}, it is quite powerful when it applies.

The starting point for Chekanov's linearization technique is an \dfn{augmentation}, which is a degree $0$ DGA map $ \aug: (\alg_\leg, \df_\leg) \to (\ff_2,0)$.  After a change of coordinates on $\alg_\leg$ defined by $\eta^\aug(a) = a + \aug(a)$, the new differential $\df^\aug = \eta^\aug \df_\leg (\eta^\aug)^{-1}$ has the property that its linear part $\df^\aug_1: A_\leg \to A_\leg$ satisfies $(\df^\aug_1)^2 = 0$.  We denote the {\bf linearized chain complex} by $LCC_*(\leg, \aug) = (A_\leg, \df^\aug_1)$. The homology groups of $LCC_*(\leg, \aug)$ are denoted $LCH_*(\leg, \aug)$ and are called the \dfn{linearized Legendrian Contact Homology (of $\leg$ with respect to \aug)}.  One may similarly define the {\bf linearized cochain complex} $LCC^*(\leg, \aug)$ with linearized codifferential $d^\aug$ and cohomology groups $LCH^*(\leg, \aug)$.

In the presence of a decomposition of the Legendrian submanifold into a link $\leg \cup \tleg$, the Reeb chords can be partitioned into \dfn{pure chords} that begin and end on the same component and \dfn{mixed chords}, with the Reeb chords from $\tleg$ to $\leg$ denoted $\mathcal{R}_{\leg \leftarrow \tleg}$. In this setting, the linearized Legendrian contact cohomology has additional structure \cite{kirill}.  We may form an augmentation $\aug_\cup$ for $\alg_{\leg \cup \tleg}$ by using augmentations $\aug$ and $\taug$ for $\leg$ and $\tleg$, respectively, on pure chords and then defining $\aug_\cup$ to be zero on the remaining mixed chords. If we let
$A_{\leg \leftarrow \tleg} \subset A_{\leg \cup \tleg}$ denote the $\ff_2$ vector space generated by $\mathcal{R}_{\leg \leftarrow \tleg}$, then it is not hard to see that the restriction of the linearized codifferential 
$d^{\aug_\cup}$ to $A_{\leg \leftarrow \tleg}$ 
yields a subcomplex of $LCC^*\left(\leg \cup \tleg, {\aug_\cup}\right)$, which we will denote
by $LCC^*\left((\leg, \aug) \leftarrow (\tleg, \taug)\right)$.  Alternatively, we may use the moduli space
from Definition~\ref{defn:lch-mixed-ms} to  directly define the codifferential of 
$LCC^*\left((\leg, \aug) \leftarrow (\tleg, \taug)\right)$ 
on a Reeb chord $\mixed{c}$ from $\tleg$ to $\leg$ by:

\begin{equation}
	d^{\aug,\taug}(\mixed{c}) = \sum_{\substack{\dim \ms^{\mathbf{J}}_{\leg,\tleg}(\mixed{a}; \mathbf{b}, \mixed{c}, \tb) = 0 \\    
	 \aug(\mathbf{b}) = 1 = \taug(\tb)} } \#\ms^{\mathbf{J}}_{\leg,\tleg}(\mixed{a}; \mathbf{b} , \mixed{c},  \tb )\, \mixed{a}.
\end{equation}

 \begin{rem} \label{rem:induced-aug} As shown in \cite{ehk:leg-knot-lagr-cob},
given an exact Lagrangian cobordism $L$ from $\leg_{-}$ to $\leg_{+}$,  there
is a DGA  morphism $\Phi_{L}: (\mathcal A(\leg_{+}), \partial_{+}) \to (\mathcal A(\leg_{-}), \partial_{-})$.  As a consequence, an augmentation
 $\aug_{-}$   of $\leg_{-}$ induces an $\aug_{+} = \aug_{+}(L, \aug_{-})$ of $\leg_{+}$ by
\begin{equation} \label{eqn:induced-aug}
\aug_{+} = \aug_{-} \circ \Phi_{L}.
\end{equation}
 
 \end{rem}

\subsection{The Fundamental Class}
\label{ssec:fundamental-class}
We now will explain some important long exact sequences that involve $LCH^{*}(\leg, \aug)$ and how the constuction of maps in this sequence
implies the existence of particular $J$-holomorphic curves.

An important structural result for the linearized Legendrian Contact Homology of a {\it horizontally displaceable} Legendrian is the duality long exact sequence \cite{high-d-duality}:\footnote{See also \cite{duality} for a precursor in $\rr^3$ and \cite{ bst:construct, josh-lisa:obstr} for a similar result for generating family homology.}

\begin{equation} \label{eq:duality}
	\xymatrix{
	 \cdots \ar[r] & LCH_k(\leg,\aug) \ar[r]^-{\rho_*} & H_k(\leg) \ar[r]^-{\sigma_*} & LCH^{n-k}(\leg, \aug) \ar[r] & \cdots
	}
\end{equation}
For a  horizontally displaceable Legendrian $\leg$, we define the \dfn{fundamental class} 
\begin{equation} \label{eqn:fun-class}
\lambda = \lambda_{\leg, [m], \aug} := \sigma_{*}[m] \in LCH^n(\leg,\aug)
\end{equation}
 for  $[m]$ a generator of $H_0(\leg)$.  It was shown in \cite{high-d-duality} that when $\leg$ is {\it connected}, the map $\sigma_*$ is injective on $H_0(\leg)$
 and thus $\lambda$ is non-zero.  When one examines the construction of $\sigma_{*}$ at the chain level, one sees that the non-triviality of the
 fundamental class implies the existence of a  $J$-holomorphic curve, for $J \in \mathcal{J}^{cyl}_\pi$, that passes through an arbitrary
  point $m \in L = \rr \times \leg$.  

For our ultimate goal of studying relative embeddings into Lagrangian cobordisms, we will need to work in a more general setting.  The long exact sequence in (\ref{eq:duality}) has been generalized to Lagrangian cobordisms:

\begin{thm}[Generalized Duality, Theorem 1.2 of \cite{c-dr-g-g-cobordism}]\label{thm:duality-general}
Suppose $L$ is a Lagrangian cobordism from $\leg_{-}$ to $\leg_{+}$, $\aug_{-}$ is an augmentation of $\leg_{-}$
and $\aug_{+}$ is  the augmentation of $\leg_{+}$ induced by $L$ from $\aug_{-}$.   If $\leg_{-}$ is horizontally displaceable,   there
is a long exact sequence
\begin{equation} \label{eq:gen-duality}
	\xymatrix{
	 \cdots \ar[r] & LCH_k(\leg_{-},\aug_{-}) \ar[r]^-{\Rho_*} & H_k(L) \ar[r]^-{\Sigma_*} & LCH^{n-k}(\leg_{+}, \aug_{+}) \ar[r] & \cdots
	}
\end{equation}
\end{thm}

The generalized fundamental class will be defined as $\Sigma_{*}[m]$ for $[m]$ a generator of $H_{0}(L)$.
The fact that the generalized fundamental class does not vanish will imply the existence of a  $J$-holomorphic curve passing through an arbitrary point $m \in L$.  To see why, we need to understand the chain-level construction of the map $\Sigma_*$.

The first step in describing the map $\Sigma_{*}$ is to specify the geometric setting, which we take from \cite[\S 7.2 and \S8.4.2]{c-dr-g-g-cobordism}.  Given $L$  as in the statement of Theorem~\ref{thm:duality-general}, we find a small perturbation $\tL$ of $L$ as follows. We emphasize that this perturbation depends on the point $m \in L$.  

First let $h: \rr \to \rr$ be a smooth function that satisfies:
\begin{enumerate}
\item $h(s) = -e^{s}$ for $s \leq u_{-} <  s_{-}$, 
\item $h(s) = e^s - T$ for $s \geq u_{+} >  s_{+} = 0$ for some $T>0$, and
\item $\left(\frac{dh}{ds}\right)^{-1}(0)$ is a connected interval containing $[s_{-},s_{+}]$.
\end{enumerate}
A picture of such a function $h$ appears in \cite[{Figure 12}]{c-dr-g-g-cobordism}.
 Let $L^{h}$ be the image of $L$ under an time-$1$ Hamiltonian flow associated to $\epsilon h$, for a sufficiently small $\epsilon$, and let $\leg^{h}_{\pm}$ be the images of the Legendrians $\leg_{\pm}$. 
 Note that $\leg^{h}_{\pm}$ is simply a small shift of $\leg_{\pm}$ in the  
 $\pm$ Reeb direction, where $\epsilon$ is chosen so that the shift is smaller than the shortest Reeb chord of $\leg_{\pm}$. To obtain isolated Reeb chords between the Legendrians at the ends and isolated intersection points between the compact portions of the Lagrangians, we further modify $L^{h}$ to $\tL$ in two steps.  For both perturbations, we use a particular Weinstein neighborhood of $L^{h}$ as constructed in \cite[\S7.2]{c-dr-g-g-cobordism}:
we symplectically identify a neighborhood $N$ of $L$ with a neighborhood $N_{0}$ of the $0$-section in $T^*L$ in such a way that when $L$ coincides with the cylinders
$\rr \times \leg_{\pm}$, $N_{0}$ coincides with a neighborhood of the $0$-section in $T^{*}(\rr \times \leg_{\pm})$, which in turn can be identified with
$\rr \times V$ where $V$ is  a neighborhood of the $0$-jet in $J^{1}\leg_{\pm}$.
With this identification, we first construct $L^{h,f}$, a non-compact perturbation of the cylindrical ends of $L^{h}$.  Let $f_\pm: \leg_\pm \to (0,\delta]$ be small, positive Morse functions. These Morse functions may be used to construct perturbations of  $\leg_{\pm}$ by taking the $1$-jets $\pm j^1 f_\pm \subset J^1\leg_\pm$. We construct $L^{h,f}$ by cylindrically extending $\leg^{h,f}_{\pm}$. Finally, we construct $\tL$ from $L^{h,f}$ by taking a particular $\delta$-small compactly supported perturbation of $L^{h,f}$ so that  
{\it 
$\tL$ is the graph of $dF$ for a Morse function $F$ that has a unique local minimum at the point $m \in L \cap \tL$. } 
 
 With the geometric background in place, we may make the following identifications:

\begin{prop}[\cite{c-dr-g-g-cobordism}, {Proposition 7.5 and Theorem 7.9}] \label{prop:identify-complexes}
Given a Lagrangian cobordism $L$ and a Lagrangian cobordism $\tL$ constructed from $L$ using functions $\epsilon h,f, F$ as above, we have:
\begin{enumerate}
\item $LCH^*\left( \left(\leg_{+}, \aug_{+} \right) \leftarrow (\leg_{+}^{h,f}, \aug_{+} ) \right) \simeq LCH^*\left(\leg_{+}, \aug_{+}\right)$, and
\item $FH^*(L, \tL) \simeq MH_{n+1-*}(F) \simeq H_{n+1-*}(L)$.  
\end{enumerate}
\end{prop}

In the first identification, we have identified augmentations on $\leg_{+}$ and $\leg_{+}^{h,f}$  through a canonical {bijection} of Reeb chords  as in \cite[Remark 7.6]{c-dr-g-g-cobordism}.
In the second identification, $MH_{*}(F)$ refers to the Morse homology.

It follows that we may construct $\Sigma_*: H_0(L) \to LCH^{n}(\leg_{+}, \aug_{+})$ on the cochain level as a map 
$$\Sigma = \Sigma_{L, \aug_{-}}: FC^{n+1}(L,\tL) \to LCC^{n}\left((\leg_{+}, \aug_{+}) \leftarrow (\leg_{+}^{h,f},\aug_{+} ) \right).$$  
Such a map is defined as the 
$d_{+0}$ map in the Cthulu complex described in \cite[\S 4.1]{c-dr-g-g-cobordism}: for $m \in L \cap \tL$ corresponding to the unique local minimum of $F$,
we define $d_{+0}$ on the $(n+1)$-cochain $m$ by 
using $\mathbf{J}$ and the Floer to LCH moduli space defined in Definition~\ref{defn:floer-lch-ms}:
\[d_{+0}(m) :=   \sum_{\substack
{\dim \ms^{\mathbf{J}}_{L \leftarrow \tL}(\mixed{a};\mathbf{b},m, \tb) = 0 \\ \aug_{-}(\mathbf{b})  =  1 
	=  \taug_{-}(\tb)      }
	}
 \# \mathcal M_{L \leftarrow \tL}^{\mathbf J} (\mixed{a}; \mathbf{b}, m, \tb ) 
\cdot a,\]
 where $\taug_{-}$ is again the augmentation of $\tleg_{-}$ induced by $\aug_{-}$ as explained in \cite[Remark 7.6]{c-dr-g-g-cobordism}.

Following \cite[\S 8.4]{c-dr-g-g-cobordism}, 
 the {\bf fundamental class} induced by $L$, $[m] \in H_{0}(L)$, 
and the augmentation $\aug_{-}$ of $\leg_{-}$ is defined to be the image
  \begin{equation} \label{eqn:fun-class-general}
 \tlambda =   \tlambda_{L, [m], \aug_{-}}  := \Sigma_{*}([m]) \in LCH^{n}(\leg_{+}, \aug_{+}).
\end{equation} 
Under the assumption that  $L$ is {\it connected}, \cite[Proposition 8.7]{c-dr-g-g-cobordism} shows that
  $\tlambda$ agrees with the fundamental class of the Legendrian at the positive end given in  Equation (\ref{eqn:fun-class}):
\[\tlambda_{L, \aug_{-}} = \lambda_{\leg_{+}, \aug_{+}} \in LCH^{n}(\leg_{+}, \aug_{+}).\]

 We gather the conditions necessary for a nonvanishing fundamental class in the following definition:
   
 \begin{defn}  \label{defn:fun-cobord}  A Lagrangian cobordism $L$ from $\leg_{-}$ to $\leg_{+}$ (satisfying conditions
 in Definition~\ref{defn:cobordism}) is a \dfn{\cobord}   if:
 \begin{enumerate}
 \item $\leg_{-}$ and $\leg_{+}$ are both horizontally displaceable;
 \item each component of  $\leg_{-}$ admits an augmentation; and
 \item $\leg_{+}$ is connected.
  \end{enumerate}
 \end{defn}

Observe that we also want $L$ to be connected, however this is implied by the other conditions:
\begin{lem}  \label{lem:fun-cobord-connected} Any \cobord\ is connected.
\end{lem}

\begin{proof} Since $\leg_{+}$ is connected, a non-connected Lagrangian cobordism would only be possible
if $\leg_{-}^{*} \subset \leg_{-}$, consisting of a component or a union of components of $\leg_{-}$, admits a Lagrangian cap, which is a Lagrangian cobordism from $\leg_{-}^{*}$ to $\emptyset$.
By a result of Dimitroglou-Rizell, \cite[Corollary 1.9]{rizell:uniruled},   the DGA of $\leg_{-}^{*}$ would be acyclic, and from this it is easy to verify that 
$\leg_{-}^{*}$  cannot admit an augmentation, contradicting the fact that each component of $\leg_{-}$ admits an augmentation.  
 \end{proof}

The upshot of this discussion is the following lemma:

\begin{lem}
	If $L$ is a \cobord, then $\tlambda_{L, \aug_{-}}  \neq 0$.
\end{lem}

 The nonvanishing of the fundamental class implies the following existence theorem of $J$-holomorphic curves: 
 
\begin{cor} \label{cor:uniruling} Suppose $L$ is  a fundamental cobordism, $\aug_{-}$ is an augmentation of $\leg_{-}$, and 
$\aug_{+}$ is the augmentation of $\leg_{+}$ induced by $L$. 
For any $m \in L$ and any generic path $\mathbf{J}$, there exists:
\begin{enumerate}
\item a perturbation $\tL$ of $L$, as constructed above from functions $\epsilon h, f_\pm$, and $F$,  with $m \in L \cap \tL$,
\item  a representative $a_1 + \cdots + a_k$ (depending on $m, L, \aug_{-}$) of the fundamental class $\lambda_{\leg_{+}, \aug_{+}}$
that is defined using the perturbation $\tL$, and
\item  for each summand $a_{i}$, $i \in \{1, \ldots, k\}$, of this representative of the fundamental class, a $\mathbf{J}$-holomorphic curve $u \in \ms^{\mathbf{J}}_{L \leftarrow \tL}(\mixed{a}_i;\mathbf{b}, m, \tb)$
with $\aug_{-}(\mathbf{b}) = 1 = \taug_{-}(\tb)$,  where $\mixed{a}_{i} \in \mathcal R_{\leg_{+} \leftarrow \leg_{+}^{h,f}}$  is the Reeb chord corresponding to $a_{i} \in \mathcal R_{\leg_{+}}$.
\end{enumerate}

\end{cor}

\section{The Fundamental Capacity}
\label{sec:capacity}

In this section, we will show that  the areas of the  $J$-holomorphic curves whose existence is guaranteed by Corollary~\ref{cor:uniruling}  
are bounded above by the  ``fundamental capacity'' of the Legendrian submanifold.  This capacity is a type of spectral invariant that was defined in \cite[\S 4.3]{josh-lisa:cob-length}, inspired by Viterbo's work \cite{viterbo:generating} on generating families and reminiscent of Hutchings' Embedded Contact Homology capacities \cite{hutchings:ech-cap}.

\subsection{Area and Energy}

Typically, a notion of energy for the $J$-holomorphic curves used to construct the Legendrian Contact Homology 
is defined with an eye to proving that the sum in Equation (\ref{eq:diffl}) is 
 finite.  To define such an energy in the setting of a Lagrangian cobordism $L$ that is cylindrical outside of $[s_{-},s_{+}] \times \cont{P}$, we use a notion of ``Lagrangian energy,'' which is a close relative to an energy introduced in \cite{c-dr-g-g}; {we use the form specified in \cite[\S3.4]{c-dr-g-g-cobordism}.}

Suppose that, for some small $\delta>0$, the cobordism $L$ is cylindrical outside of $[s_-+\delta, s_+ - \delta]$. Next, consider a smooth, increasing function
\[\varphi(s) = \begin{cases} e^{s_{-}}, & s \leq s_{-} \\ e^s, & s_{-}+\delta \leq s \leq s_{+}-\delta \\ e^{s_{+}}, & s \geq s_{+}. \end{cases}\]

\begin{defn} \label{defn:area-energy}
Given a curve $u \in \ms^J_\leg(a;\mathbf{b})$ or $u \in \ms^{\mathbf{J}}_{L \leftarrow \tL}(\mixed{a};\mathbf{b}, m, \tb)$
and a region $R_c^d = u^{-1}\left([c, d] \times J^1M\right)$, we define the \dfn{$[c,d]$-area of $u$} by
\[A_c^d(u) = \int_{R_c^d} u^*d(e^s \alpha)\] 
and the \dfn{$L$-energy of $u$} by
\[E_c^d(u) = \int_{R_c^d} u^* d(\varphi(s) \alpha).\]
\end{defn}

In order to bound the area of a $J$-holomorphic curve in the symplectization by a quantity computable from a Legendrian at the positive end, we introduce the actions of Reeb chords and intersection points.  Suppose that $L$ and $\tL$ are Lagrangian cobordisms as constructed in Section~\ref{ssec:fundamental-class} with positive ends at the Legendrians $\leg_{+}$ and $\tleg_{+}$, respectively.  Let $\rho$ and $\trho$ denote primitives of the pullbacks of $e^s \alpha$ on $L$ and $\tL$ with constant values, respectively, $C_{+}$ and $\tC_{+}$ at the positive end (and constant values $C_{-} = 0 = \tC_{-}$ at the negative ends). 
 Recall the height of a Reeb chord, $h(a)$,  from Equation~(\ref{eqn:height}).  We define the following actions as in \cite[Section 3.4]{c-dr-g-g-cobordism}:
 \begin{defn}   \label{defn:actions}  The {\bf action} of
  \begin{enumerate}
 \item a pure Reeb chord $a \in \mathcal R_{\leg_{\pm}}$ is  
 $$\mathfrak{a}(a) = e^{s_\pm} h(a);$$
\item  a mixed Reeb chord $\mixed{a} \in \mathcal R_{\leg_{+} \leftarrow \tleg_{+}}$  is  
$$\mathfrak{a}(\mixed{a}) = e^{s_+} h(\mixed{a}) + {(\tC_{+}-C_{+})};$$

\item an intersection point $m$ of $L$ and $\tL$  (where the holomorphic curve jumps from $L$ to $\tL$) is  
$$\mathfrak{a}(m) = { \trho(m)-\rho(m)}.$$
\end{enumerate}
 \end{defn}
 
We may control the actions of mixed Reeb chords and intersection points between $L$ and $\tL$ using the following lemma:
 
 \begin{lem}\label{lem:mixed-intersection-action} Given an exact Lagrangian cobordism $L$ from $\leg_{-}$ to $\leg_{+}$ with primitive $\rho$,  $m \in L_{s_{-}}^{s_{+}}$, and an arbitrary $\delta > 0$,
  it is possible to construct an exact Lagrangian perturbation $\tL$ with primitive $\trho$ as
 in Section~\ref{ssec:fundamental-class} that satisfies the following conditions:
 \begin{enumerate}
 \item If $\mixed{a} \in \mathcal R_{\leg_{+} \leftarrow \leg_{+}^{h,f}}$ and $a \in \mathcal R_{\leg_{+}}$ are identified as in \cite[Remark 7.6]{c-dr-g-g-cobordism},
  then 
  \[| \mathfrak a (\mixed{a}) - \mathfrak a(a) | < \delta.\]
  
\item At the intersection point $m \in L \cap \tL$, we have:
  $$|\mathfrak{a}(m)| < \delta.$$
  \end{enumerate}
 \end{lem}
 
\begin{proof}  
The construction of $\tL$ from $L$ in Section~\ref{ssec:fundamental-class} may be thought of as finding a Hamiltonian isotopy $\tau_t$ that carries $L$ to $\tL$ at time $1$.  The isotopy is generated by a Hamiltonian $H_t$ that is $C^1$-small on $(-\infty, u_+] \times \cont{P}$; to see this, note that since the $C^1$ norm of the function $h$ is bounded on $(-\infty, u_+] \times \cont{P}$, we have that $\epsilon h$ is $C^1$-small, and the functions $f$ and $F$ were chosen to be $C^1$ small from the beginning.  

A direct calculation, see for example \cite[Proposition 9.18]{mcduff-salamon}, shows that
\[\tau_1^{*}(e^{s}\alpha) - e^{s}\alpha = dG,\] where 
\[G = \int_{0}^{1} (i_{X_{t}} e^{s}\alpha - H_{t}) \circ \tau_t \, dt.\]
 Since $H_{t}$ is $C^{1}$-small on $(-\infty, u_+] \times \cont{P}$, we see that the function
 $G$ has a small $C^0$ norm on that set after adjusting by a constant so that $G \equiv 0$ for sufficiently small $s$.  Another direct calculation shows that if $\rho$ is a primitive for $e^s \alpha$ along $L$, then  $(\rho + G)\circ \tau_1^{-1}$ is a primitive for $e^s \alpha$ along $\tL$. Since both $\rho$ and $(\rho + G)\circ \tau_1^{-1}$ are constant for $s \geq u_+$, we see that the primitives of $e^s\alpha$ along $L$ and $\tL$ are $C^0$ close.  Thus,  both claims in the lemma follow.
  \end{proof}

We can bound the $[c,d]$-area of a curve $u$ in terms of actions:

\begin{prop} \label{prop:area-reeb}  Assume $c < d \leq s_{+}$.  For any curve 
 $u \in \ms^{\mathbf{J}}_{L \leftarrow \tL}(\mixed{a};\mathbf{b}, m, \tb)$, we have:
	\[A_c^d(u) < \mathfrak{a}(\mixed{a}) - \mathfrak{a}(m).\]
\end{prop}

\begin{proof} When $s \leq d \leq s_{+}$, $e^{s} \leq \varphi(s)$, and thus we see $A_c^d(u) \leq E_c^d(u)$.  By the compatibility of $J$ with $d(e^{s}\alpha)$
and fact that $d < \infty$, we see that 
\[E_c^d(u) < \int_{D_{k}} u^* d(\varphi(s) \alpha).\]
The strict inequality above comes from the fact that $R_{c}^{d}$ differs from  $D_{k}$ by a set of positive measure on which the integrand is strictly positive.
By Stokes' formula,  for 
$u \in \ms^{\mathbf{J}}_{L \leftarrow \tL}(\mixed{a};\mathbf{b}, m, \tb)$, we compute:
\[\int_{D_{k}} u^* d(\varphi(s) \alpha) = \mathfrak a(\mixed{a}) - \mathfrak a(m) - \sum_{b_{i}} \mathfrak a(b_{i}) -  \sum_{\tilde b_{i}}\mathfrak a(\tilde b_{i})
\leq \mathfrak a(\mixed{a}) - \mathfrak a(m).\]
The proposition follows.
\end{proof}

\subsection{Filtrations and Capacities}
\label{ssec:capacity}

In order to apply the bound in Proposition~\ref{prop:area-reeb} to the $J$-holomorphic curves guaranteed by Corollary~\ref{cor:uniruling}, we need to refine the Legendrian Contact Homology framework using an energy filtration.  In particular, we will define the fundamental capacity as in \cite{josh-lisa:cob-length}. We begin with a filtration on $A_\leg$ with respect to the height of the generating Reeb chords:  for any $w \in \rr$, define 
\[\mathcal{R}^w_\leg = \left\{ a \in \mathcal{R}_\leg \;:\; h(a) \geq w\right\}.\]
Let $F^w A^*_\leg$ be the graded vector space generated by $\mathcal{R}^w_\leg$; energy considerations show that it is, in fact, a subcomplex of $(A^*_\leg, d^\aug)$.  We then define the \dfn{Filtered Linearized Legendrian Contact Cohomology} $LCH^*_w(\leg, \aug)$ to be the homology groups of the quotient $A_\leg^* / F^wA^*_\leg$.

Let $p^w: LCH^*(\leg,\aug) \to LCH^*_w(\leg,\aug)$ be the map induced by the projection from $A^*_\leg$ to $A^*_\leg /F^wA^*_\leg$.  It is straightforward to check that for $w$ close to $0$, $p^w$ is the zero map, while for sufficiently large $w$, $p^w$ is an isomorphism.  Thus, we define:

\begin{defn} \label{dfn:capacity}
	Given a connected, horizontally displaceable Legendrian submanifold $\leg \subset \cont{P}$, an augmentation $\aug$, and its fundamental class 
	$\lambda_{\leg, \aug}$, the \dfn{fundamental capacity} $c(\leg, \aug)$ is defined to be:
\[ c(\leg, \aug) = \sup \{w \in \rr \;:\; p^w(\lambda_{\leg,\aug}) = 0\}.\]
\end{defn}

We know that $c(\leg, \aug)$ is always the height of a Reeb chord of $\leg$ \cite[Lemma 4.7]{josh-lisa:cob-length}.
Specifically, for each $x \in A_\leg$ that represents the fundamental class with  $a_x$ the  Reeb chord of minimal height with nonzero coefficient in $x$, then 
\begin{equation} \label{eq:cap-chord}
	c(\leg, \aug) = \max \{ h(a_x) \;\vert\; x \text{ represents } \lambda\}.
\end{equation}

We use {Equation~(\ref{eq:cap-chord})}, Lemma~\ref{lem:mixed-intersection-action}, and Proposition~\ref{prop:area-reeb} to refine Corollary~\ref{cor:uniruling} as follows:
 
 \begin{cor} \label{cor:fun-disk}   {  Let $L$ be a \cobord,  $\aug_{-}$ an augmentation of $\leg_{-}$, and $\aug_{+}$ the induced augmentation of $\leg_{+}$.  For any $\epsilon>0$, $m \in L$, and generic path $\mathbf J$, there exists a perturbation $\tL$, a mixed chord $\mixed{a}$ corresponding to the shortest chord in the corresponding representative of the fundamental class, and a $\mathbf{J}$-holomorphic curve $u \in \ms^{\mathbf{J}}_{L \leftarrow \tL}(\mixed{a};\mathbf{b}, m, \tb)$ with $\aug_-(\mathbf{b}) = 1 = \aug_-(\tb)$ such that for any $[c,d] \subset (-\infty, s_+]$, we have
 	\[ A^d_c(u) \leq e^{s_+} c(\leg_+, \aug_+) + \epsilon.\]
}

 \end{cor}

We will call a curve $u$ described by Corollary~\ref{cor:fun-disk} a {\bf fundamental $J$-holomorphic disk}.

\section{Upper Bounds on the Relative Gromov Width}
\label{sec:upper}

The goal of this section is to prove Theorem~\ref{thm:ub}, which provides an upper bound on the relative Gromov width.  Following ideas from the classical argument of Gromov \cite{gromov:hol} that were adapted to the relative setting by Barraud and Cornea
 \cite{bc:lagr-intersection}, we use the $J$-holomorphic disks from 
 Corollary~\ref{cor:fun-disk} to obtain the desired upper bounds.  More specifically, given an embedding of a ball,
we will pull back the $J$-holomorphic curve passing through the center of the ball to a holomorphic (and hence minimal) surface in $\cc^n$.  We will then apply
monotonicity of area for minimal surfaces  to compare the area of the holomorphic curve to the capacity of the ball. 

\subsection{Monotonicity}
\label{ssec:mono}

We will need a modification of the classical monotonicity property for minimal surfaces with (carefully controlled) boundary, adapted from arguments of Ekholm, White, and Wienholz \cite{eww:minimal}. The key quantity in monotonicity is the density function  $\Theta(t)$, which is the
 ratio of the areas of $\Sigma$, the image of a holomorphic curve with boundary on a union of Lagrangian planes, 
 to those of a plane $\rr^2$ inside the ball of radius $t$:
\[\Theta(t) = \frac{\operatorname{Area}(\Sigma \cap B^{2n}(t))}{\pi t^2}.\]

\begin{thm} \label{thm:mono}  Let $P, \tP \subset (\rr^{2n}, \omega_0)$ denote transverse
Lagrangian planes that pass through the origin, and let
	$\Sigma \subset \rr^{2n}$ be the image of a proper $J_0$-holomorphic curve with $0 \in \partial \Sigma \subset P \cup \tP$.  Then
	   $\Theta(t)$ is non-decreasing.  In particular, if we let $Z = \lim_{t \to 0} \Theta(t)$, 
	then for any $t>0$, we have
	\[ \pi t^2 \leq \frac{\mathrm{Area}(\Sigma \cap B^{2n}(t))}{Z}.\]
 \end{thm}

\begin{proof}  By a result of Ahn \cite[Proposition 3.1]{ahn:boundary-id},  we know that $\Sigma':= \Sigma \setminus (P \cap \tP)$ is smooth.
  In particular, since $\Sigma$ is the image of a proper map, the length of $\partial \Sigma' \cap B^{2n}(t)$  is finite for all $t$.
Then, since $\Sigma'$ is a smooth, minimal $2$-manifold in $\rr^{2n}$ with boundary of finite length, the arguments in 
the proof of \cite[Theorem 9.1]{eww:minimal} show that 
	\[\frac{d}{dt}\Theta(t) = \frac{d}{dt}\int_{\Sigma' \cap B^{2n}(t)} \frac{\left|D^\perp |x|\right|^2}{|x|^2}\,dA - \frac{1}{t^3} \int_{\partial \Sigma' \cap B^{2n}(t)} x \cdot n_{\Sigma'}\, ds,
	\] 
	where $n_{\Sigma'}$ is the outward-pointing normal along $\partial \Sigma'$ and $D^\perp |x|$ denotes the projection of the derivative of $|x|$ to the orthogonal complement of $T_x\Sigma'$. Since $\Sigma'$ is the image of a $J_0$-holomorphic curve and $P, \tP$ are Lagrangian,  a lemma of Ye \cite[Lemma 2.1]{ye:gr-compactness} shows that for all $x \in \partial \Sigma'$,  $n_{\Sigma'}(x) \perp T_x(P \cup \tP)$.
Further, as $P, \tP$ are Lagrangian planes through the origin, we have that $x \in T_x(P \cup \tP)$.  Thus, our second integrand vanishes, and since the first integrand is non-negative, we obtain
$\frac{d}{dt}\Theta(t) \geq 0$.  The theorem follows.
	\end{proof}

\subsection{Upper Bound}\label{ssec:ub}

With the appropriate form of monotonicity in hand,  we are ready to prove Theorem~\ref{thm:ub}.

\begin{proof}[Proof of Theorem~\ref{thm:ub}]   To set notation, let $L$ be a \cobord from $\leg_{-}$ to $\leg_{+}$, $\aug_{-}$ an augmentation of $\leg_{-}$, and $\aug_{+}$  the augmentation of $\leg_{+}$ induced by $L$.  Suppose that we have a relative symplectic embedding  $\psi: B^{2n}(r) \hookrightarrow ((-\infty,0] \times \cont{P}, L_{-\infty}^0)$.  It suffices to show that, for arbitrary $\epsilon>0$, we have 
\begin{equation} \label{eqn:ub-rough}
\pi r^2 \leq 2 (c(\leg_+,\aug_+) + \epsilon).
\end{equation}

Fix $\epsilon>0$.  Observe that by carefully choosing the perturbation function $F$ near $\psi(0)$, we can construct $\tL$ so that the pullback of $\tL$ to $B^{2n}(r)$ is the Lagrangian plane $\tP$ given by the graph of the linear map $\delta\cdot\operatorname{Id}$ for some small $\delta > 0$.  

Consider a domain dependent $\mathbf{J}$  induced by a path in $\mathcal{J}^{adm}$ such that each $J$ in the path extends $\psi_*J_0$.
There is a sequence of generic domain dependent complex structures $\mathbf J_k$ that converge to $\mathbf J$.   
For each $k$, Corollary~\ref{cor:fun-disk}  yields a fundamental disk $u_{k} \in \ms_{L \leftarrow \tL}^{\mathbf{J}_k}(\mixed{a}_k; \mathbf{b}_{k}, \psi(0), \tb_{k})$.  Since for all $k$, 
  there are only a finite number of options for
  $\mixed{a}_k, \mathbf{b}_{k}$ and $\tb_{k}$, by passing to a subsequence, we can assume 
   there exists a sequence $u_{k} \in  \ms_{L \leftarrow \tL}^{\mathbf{J}_k}(\mixed{a}; \mathbf{b}, \psi(0), \tb)$, i.e. the sequence lies in moduli spaces with fixed asymptotics.   Corollary~\ref{cor:fun-disk} and the fact that $s_{+} = 0$ show that 
\begin{equation} \label{eqn:area-cap-n}
A_{-\infty}^0 (u_k) \leq c(\leg_{+}, \aug_{+}) + \epsilon.
\end{equation}
Gromov compactness then yields a subsequence of the $u_k$ that converges to a $\mathbf J$-holomorphic map $u\in \ms^{\mathbf{J}}_{L \leftarrow \tL}(\mixed{a}; \mathbf{b},  \psi(0), \tb)$;  exactness and the usual SFT compactness arguments imply that no bubbling can occur inside $\psi(B^{2n}(r))$.  Thus,  Equation~(\ref{eqn:area-cap-n}) yields the bound
\begin{equation} \label{eqn:area-cap}
A_{-\infty}^0 (u) \leq c(\leg_{+}, \aug_{+}) + \epsilon.
\end{equation}

Let $\widetilde D = u^{-1}(\im \psi)$ and $\widetilde u  = \psi^{-1} \circ u|_{\widetilde D}$.  Further, let
$\Sigma$ denote the image of $\widetilde u$.  
It is straightforward to see that
$\partial \Sigma \cap \operatorname{Int} B(r) \subset \rr^n \cup \tP$, and that  $\widetilde u (z_1) = 0 \in \rr^n \cap \tP$.
Applying Theorem~\ref{thm:mono}  and the bound on the area provided by Equation~(\ref{eqn:area-cap}), we find that for all $t < r$,
\begin{equation} \label{eqn:c-z bound2}
 \pi t^2 \leq \frac{\mathrm{Area}(\Sigma \cap B^{2n}(t))}{Z} \leq \frac{c(\leg_{+}, \aug_{+})+ \epsilon}{Z}.
 \end{equation}
It remains to find $Z$, which, since $0$ is not a smooth point of $\partial \Sigma$, will require the asymptotic analysis in Robbin and Salamon \cite{robbin-salamon:hol}.

We begin the computation of $Z$ by  setting notation.   
 By a conformal change of coordinates in a neighborhood $U \subset \widetilde D$ of  $z_1$, we can describe points in
 $S := U \setminus \{z_1\}$ by $ \sigma + i \tau$, with   $\sigma \in [0,\infty)$, $\tau \in [0,1]$.  Let $\partial_0 S = [0,\infty)$ and $\partial_1 S = [0,\infty) + i$.  The construction of $\widetilde u$ implies that $\widetilde u(\partial_0 S) \subset \rr^n$, while $\widetilde u(\partial_1 S) \subset \tP$. 
We next apply Theorem B (or, more accurately, a coordinate-by-coordinate application of Theorem C) of \cite{robbin-salamon:hol}
to get an asymptotic expression for $\widetilde u$. 
In particular, since the counter-clockwise angle from $\rr^n$ to $\tP$ in each coordinate is $\delta \in (0, \pi)$, we obtain a unique nonzero complex vector $v \in \tP$, a positive real number $\beta = k \pi -  \delta$  for some positive integer $k$, and a $\gamma>0$ such that 
\begin{equation*}
	\widetilde u(\sigma + i \tau) = v e^{-\beta(\sigma+ i \tau)} + O(e^{-(\beta + \gamma)\sigma}).
\end{equation*} 
Hence, as we let $\sigma \to \infty$ --- which is the same as letting $t \to 0$ --- we see that $\Sigma$ asymptotically covers a fraction $\frac{k\pi - \delta}{2\pi}$ of the area of tangent disk to $\Sigma$ at the origin.  That is, we obtain 
\begin{equation*}
	Z = \lim_{t \to 0} \Theta(t) = \frac{k\pi - \delta}{2\pi}.
\end{equation*}
Equation~(\ref{eqn:c-z bound2}) then implies that for all $t < r$,
\[ \pi t^2 \leq \frac{2 \pi (c(\leg_{+}, \aug_{+})+ \epsilon)}{k\pi - \delta} \leq \frac{2 \pi (c(\leg_{+}, \aug_{+})+ \epsilon)}{\pi - \delta}.\]
Since $\delta>0$ was arbitrary, we obtain the desired inequality in Equation~(\ref{eqn:ub-rough}).  Theorem~\ref{thm:ub} follows.
\end{proof}

\bibliographystyle{amsplain} 
\bibliography{main}

\end{document}

%% file: header.tex
\newcommand{\df}{\ensuremath{\partial}}
\newcommand{\ms}{\ensuremath{\mathcal{M}}}

\def\TM+{T^*(\rr_+ \times M)}
\def\rp{\mathbb R_+}

\newcommand{\cc}{\ensuremath{\mathbb{C}}}
\newcommand{\rr}{\ensuremath{\mathbb{R}}}

\newcommand{\ff}{\ensuremath{\mathbb{F}}}

\theoremstyle{plain}
\newtheorem{thm}{Theorem}[section]
\newtheorem{cor}[thm]{Corollary}
\newtheorem{lem}[thm]{Lemma}

\newtheorem{prop}[thm]{Proposition}

\theoremstyle{definition}
\newtheorem{defn}[thm]{Definition}

\theoremstyle{remark}
\newtheorem{rem}[thm]{Remark}
\newtheorem{ex}[thm]{Example}

\numberwithin{equation}{section}